\documentclass[10pt]{amsart}

\usepackage{amssymb}

\theoremstyle{plain}      \newtheorem{proposition}{Proposition}
\theoremstyle{plain}      \newtheorem{theorem}{Theorem}
\theoremstyle{remark}     \newtheorem{example}{Example}

\title[Projective distance and $g$-measures]{Projective distance and $g$-measures}
\author{L. Trejo-Valencia}

\author{E. Ugalde}
\address{Instituto de F\'isica, Universidad Aut\'onoma de
San Luis Potos\'i, Avenida Manuel Nava 6, Zona Universitaria,
78290 San Luis Potos\'\i , M\'exico.}

\date{\today}

\begin{document}

\begin{abstract}
We introduce a distance in the space of fully-supported probability measures on 
one-dimensional symbolic spaces. 
We compare this distance to the $\bar{d}$-distance and we prove that 
in general they are not comparable. Our projective distance is inspired on Hilbert's 
projective metric,  and in the framework of $g$-measures, it allows to assess the continuity 
of the entropy at $g$-measures satisfying uniqueness. It also allows to relate the speed 
of convergence and the regularity of sequences of locally finite $g$-functions, to the
preservation at the limit, of certain ergodic properties for the associate $g$-measures. 
\end{abstract}

\maketitle

%##################################################################
%##################################################################
\bigskip

\section{Introduction.}\label{sec:introduction}
%##################################################################
%##################################################################

\subsection{} In~\cite{Hilbert1885} Hilbert introduced the so called projective distance,
for which the geodesic are precisely the straight lines. It was later used by G. Birkhoff
to prove the existence and uniqueness of positive eigenvectors for positive linear 
transformations on Banach spaces~\cite{Birkhoff1957}.
Birkhoff's strategy goes as follows: uniformly positive bounded linear transformations
map the positive cone of a Banach space into itself. This transformation is 
non-expansive with respect to the projective distance, and if the image cone has finite
diameter, then the transformation is a projective contraction. In this case Banch's 
fixed point Theorem ensures the existence and uniqueness of a projective fixed point for
the linear transformation, and projective fixed points are nothing but positive
eigenvectors. Furthermore, the contractiveness ensures that the iterations of the linear 
transformation on any positive vector converge exponentially fast, in the projective 
sense, towards to the fixed point.
Birkhoff's strategy has been successfully employed in the solution of a variety of 
problems, in particular to prove existence and uniqueness of invariant measures, 
and the exponential decay of correlations of convenient observables.
This has been done for symbolic 
systems~\cite{Ferrero1979, MaumeDeschamps2001b}, for suitable one-dimensional
maps~\cite{Liverani1995b, Liverani1998}, and for general maps with some degree of 
hyperbolicity~\cite{Liverani1995a, MaumeDeschamps2001a}.

\medskip\noindent Ornstein's $\bar{d}$-distance was introduced 
in~\cite{Ornstein1973} to give a topological characterization to the Bernoulli processes. 
This distance generates a topological structure well adapted to the study of
important ergodic properties. For instance, $\bar{d}$-limits of sequences of mixing 
processes are mixing, the class of Bernoulli processes is $\bar{d}$-closed, as well 
as the class of $K$-processes. Bressaud and coauthors, in a study of Markov 
approximation to $g$-measures (chains of complete connection in their nomenclature),
found an upper bound for the speed of $\bar{d}$-convergence of the approximations 
related to the regularity of the $g$-function~\cite{Bressaud1999b}. In a related 
work~\cite{Coelho1998}, Coelho and Quas studied the $\bar{d}$-continuity of 
$g$-measures with respect to the uniform distance between $g$-functions.

\subsection{} In~\cite{Chazottes2005a} we stablished a relation between the rate of  
projective-convergence of the Markovian approximations of a one-dimensional 
Gibbs measures and the decay of correlations of the limiting Gibbs measure.  The 
result extends straightforwardly to the case on $g$-measures defined by sufficiently 
regular $g$-functions. Our technique relies on a projective comparison of the 
marginals of the approximating measures. If the potential defining the Gibbs measure
is sufficiently regular, then the finite range approximations are  sufficiently similar 
``in the projective sense'', and in this case the mixing rate of the Gibbs measure can 
be upper bounded by a function of the mixing rates of the approximations. 
Additionally, in this fast approximation regime, the entropy of the approximations
converges toward the entropy of the Gibbs measure. Furthermore, since in that case 
the relative entropy of the limiting Gibbs measure with respect to the approximations
goes to cero, then Marton's bounds~\cite{Marton1996, Marton1998} ensures the 
convergence of the approximations in $\bar{d}$-distance. In a recent 
work~\cite{Maldonado2013}, Maldonado and Salgado applied our approach to study
the approximability of Gibbs measure for two-body interactions in one dimensional 
symbolic systems. 
This technique was also used in our study of the preservation of Gibbsianness under
amalgamation of symbols~\cite{Chazottes2011a}.

\subsection{}  Despite its actual and potential applications, our notion of  
``projective convergence'' has not yet been formalized, neither its relation to 
$\bar{d}$-converges or vague convergence has been established. 
The aim of this paper is to fill this gap and to explore to which extent the 
projective convergence 
as we define it, is well adapted to study particular classes of processes. 
We consider in particular the class of $g$-measures, 
leaving for a forthcoming work the study of measures obtained by random 
substitutions for which we already have some preliminary results.
The rest of the paper is organized as follows.
The next section is devoted to the study of some general properties of the 
projective distance, particularly its relation to the vague distance and the 
$\bar{d}$-distance. In Section~\ref{sec:g-measures} we study the convergence 
of Markov approximations to a $g$-measure, the continuity of the entropy at 
$g$-measures satisfying uniqueness, and we establish a criterion for uniqueness 
based on the speed of convergence and regularity of Markov approximations.   
Section~\ref{sec:conclusions} contains some concluding remark and and perspectives.

\medskip\noindent
\subsection{\bf Acknowledgements} This work was supported by the
Mexican Government through CONACyT grant CB-2009-01-129072. It was also 
partially supported by Universidad Aut\'onoma de San Luis Potos\'{\i}, via grant 
C14-FAI-04-33.33. 
We thank {\it Laboratorio Internacional Solomon Lefschetz} the financing 
of our academic exchange with Professor Chazottes from Ecole Polytechnique.
L. Trejo-Valencia is supported by CONACyT through the Ph. D. Fellowship
332432. 
%##################################################################
%##################################################################
\section{Projective Distance}
\label{sec:projective-distance}
%##################################################################
%##################################################################

\subsection{}
Let  $A$ be a finite set, which we also called alphabet, and let $X:=A^\mathbb{N}$ 
the set of infinite $A$-valued sequences. As usual, the elements of $A$ will be called 
symbols and words the finite tuples in $A$.
Given $\pmb{x}=x_1x_2\cdots \in A^\mathbb{N}$ and natural numbers
$1\leq n\leq m$,  $\pmb{x}_n^m$ denote the word $x_n x_{n+1} \dots x_{m-1} x_m$. 
The left shift $T:A^\mathbb{N}\to A^\mathbb{N}$ is such that $(T\pmb {x})_i=x_{i+1}$ 
for all $i\in\mathbb{N}$. The pair $(X,T)$ is the full shift on the alphabet $A$.

\medskip \noindent To a word $\pmb{a}\in A^n, n\in\mathbb{N}$, we associate
the cylinder set $[\pmb{a}]:=\{\pmb{x}\in A^\mathbb{N}\: : \pmb{x}_1^n=\pmb{a}\}$.
Cylinder sets are clopen in the standard Tychonoff topology and generate
the corresponding Borel $\sigma$-algebra $\mathcal{B}(X)$. 
We denote by $\mathcal{M}(X)$ the set of all Borel probability measures on $X$ 
and by $\mathcal {M}_T(X)$ the subset of $T$-invariant probability measures. 
Both $\mathcal {M}(X)$ and $\mathcal {M}_T(X)$ are compact convex sets in vague 
topology. The vague topology can be metrized by the distance
\begin{equation}\label{eq:vague}
D(\mu,\nu):=\sum_{n\in\mathbb{N}} 2^{-n}
        \left(\sum_{\pmb{a}\in A^n}|\mu[\pmb {a}]-\nu[\pmb{a}]|\right)
\end{equation}
It is known that $\mathcal{M}(X)$ as well as $\mathcal{M}_T(X)$ are convex set, 
complete and separable in the vague topology. Furthermore, they have the structure 
of a simplex, which, in the case of $\mathcal{M}_T(X)$ implies the uniqueness of the 
ergodic decomposition~\cite{Denker}.

\medskip\noindent Given $\mu,\nu\in \mathcal{M}(X)$, a coupling between $\mu$ and 
$\nu$ is a measure $\lambda\in\mathcal{M}((A\times A)^{\mathbb{N}})$ such that for 
all $n\in\mathbb{N}$,
\[
\sum_{\pmb{b}\in A^n}\lambda[\pmb{a} \times \pmb{b}]=\mu[\pmb{a}],\
 \sum_{\pmb{a}\in A^n}\lambda[\pmb{a} \times \pmb{b}]=\nu[\pmb{b}].
 \]
Here $\pmb{a} \times \pmb{b}=(a_1b_1)(a_2b_2)\cdots (a_nb_n)\in (A\times A)^n$,
for each $\pmb{a},\pmb{b}\in A^n$. 
With $J(\mu,\nu)\subset\mathcal{M}((A\times A)^{\mathbb{N}})$ we denote the set of 
all couplings between $\mu$ and $\nu$. Ornstein's $\bar{d}$-distance is given by
\begin{equation}\label{eq:dbar}
\bar{d}(\mu,\nu)=\inf_{\lambda\in J(\mu,\nu)}
\limsup_{n\to\infty}\frac{1}{n}\sum_{k=0}^{n-1} \lambda(T^{-k}\bar{\Delta}),
\end{equation}
where $\bar{\Delta}=\{ab\in A\times A: a\neq b\}$ is the complement of the diagonal. 
Distance $\bar{d}$ makes $\mathcal{M}(X)$ a complete but non-separable topological 
space. The same holds when $\bar{d}$ is restricted to the subspace of $T$-invariant 
measures $\mathcal{M}_T(X)$ (see~\cite{Shields} for instance).

\subsection{}
Let $\mathcal{M}^+(X)\subset \mathcal{M}(X)$ be the set of fully-supported Borel 
probability measures on $X$, {\em i.e.}, $\mu\in \mathcal{M}^+(X)$ if and only if 
$\mu[\pmb{a}]>0$ for all $\pmb{a}\in\cup_{n\in\mathbb{N}} A^n$.
We define $\rho:\mathcal{M}^+(X)\times\mathcal{M}^+(X)\to\mathbb{R}^+$
by
\begin{equation}\label{eq:projective}
\rho(\mu,\nu)=\sup_{n\in {\mathbb N}}
\max_{\pmb{a}\in A^n}\frac{1}{n}\,\left|\log\frac{\mu[\pmb a]}{\nu[\pmb a]}\right|.
\end{equation}

\noindent
The function $\rho$ defines a distance on $\mathcal{M}^+(X)$ which we call 
projective distance.

\begin{theorem}\label{theo:completeness}
$\mathcal{M}^+(X)$ is a complete metric space with respect to $\rho$.
\end{theorem}

\begin{proof}
Let us first verify that $\rho$ defines a metric. Clearly  $\rho(\mu,\nu)\geq0$ for all 
$\mu, \nu\in\mathcal{M}^+(X)$, and $\rho(\mu,\nu)=0$ if and only if and only if 
$\mu[\pmb{a}]=\nu[\pmb{a}]$ for all $n\in\mathbb N$ and $\pmb{a}\in A^n$ which 
readily implies $\mu=\nu$. Now, since for all $n\in\mathbb N$ and $\pmb{a}\in A^n$ 
and each $\lambda\in \mathcal{M}^+(X)$ we have
\[
\left|\log\frac{\mu[\pmb{a}]}{\nu[\pmb{a}]}\right|
   =\left|\log\frac{\mu[\pmb{a}]\lambda[\pmb{a}]}{\nu[\pmb{a}]\lambda[\pmb{a}]}\right|
   =\left|\log\frac{\mu[\pmb{a}]}{\lambda[\pmb{a}]} +
                                      \log\frac{\lambda[\pmb{a}]}{\nu[\pmb{a}]}\right|
   \leq\left|\log\frac{\mu[\pmb{a}]}{\lambda[\pmb{a}]}\right|+
                                \left|\log\frac{\lambda[\pmb{a}]}{\nu[\pmb{a}]}\right|,
\]
then $\rho(\mu,\nu)\leq\rho(\mu,\lambda)+\rho(\lambda,\nu)$ for all
$\mu,\lambda,\nu\in \mathcal{M}^+(X)$.

\medskip \noindent
Let $\mu,\nu\in\mathcal{M}^+(X)$ be such that $\rho(\mu,\nu) < \log(2)$, then all 
$n\in\mathbb N$ and $\pmb{a}\in A^n$ we have
$e^{-n\rho(\mu,\nu)}\nu[\pmb{a}] < \mu[\pmb{a}] <e^{n\rho(\mu,\nu)}\nu[\pmb{a}]$, 
which implies $|\mu[\pmb{a}]-\nu[\pmb{a}]| < (e^{n\rho(\mu,\nu)}-1)\nu[\pmb{a}]$, and
from this
\begin{equation}\label{eq:compare-vage}
D(\mu,\nu) < \sum_{n\in\mathbb{N}} 2^{-n}(e^{n\rho(\mu,\nu)}-1)=
          2\frac{e^{\rho(\mu,\nu)}-1}{2-e^{\rho(\mu,\nu)}} < \frac{4}{3}\rho(\mu,\nu).
\end{equation}
With this we prove that the vague topology is weaker than the one induced by $\rho$.

\medskip\noindent
Let us now prove that $\mathcal{M}^+(X)$ is complete with respect to the distance 
$\rho$. For this let $\{\mu_m\}_{m\in\mathbb{N}}$ be a Cauchy sequence with respect 
to $\rho$, which is a Cauchy sequence respect to $D$ as well. Since $D$ makes 
$\mathcal{M}(X)$ a complete space, then there exists $\mu\in \mathcal{M}(X)$ 
towards which $\{\mu_m\}_{m\in\mathbb{N}}$ converges.
Now, for each $n\in \mathbb{N}$, $\pmb{a}\in A^n$ and every $m\in \mathbb{N}$,
we have $e^{-n\rho(\mu_m,\mu_1)}\mu_1[\pmb{a}] \leq \mu_m[\pmb{a}]$, therefore
\[
\mu[\pmb{a}]=\lim_{m\to\infty}\mu_m[\pmb{a}] 
               \leq \mu_1[\pmb{a}] e^{-n\sup_{m\in\mathbb{N}} \rho(\mu_1,\mu_m)}>0,
\]
which proves that $\mu\in \mathcal{M}^+(X)$. Finally, since 
$\mu[\pmb{a}]=\lim_{m\to\infty}\mu_m[\pmb{a}]$, we have
\[
e^{-n\sup_{m\geq m_0}\rho(\mu_m,\mu_{m_0})}
                             \leq  \frac{\mu[\pmb{a}]}{\mu_m[\pmb{a}]} 
                             \leq e^{n\sup_{m\geq m_0}\rho(\mu_m,\mu_{m_0})}
\]
for each $n\in \mathbb{N}$, $\pmb{a}\in A^n$ and $m_0\in\mathbb{N}$. From this it
follows that
\[
\rho(\mu,\mu_{m_0})\leq \sup_{m\geq m_0}\rho(\mu_m,\mu_{m_0}),
\]
which proves that $\mu$ is the limit of $\{\mu_m\}_{m\in\mathbb{N}}$ in the projective
distance.
\end{proof}

\bigskip\noindent
As mentioned above, $\mathcal{M}(X)$ is separable in the vague topology while it is 
non-separable with respect to the topology induced by $\bar{d}$. In this respect, 
regarding the projective distance we have the following.

\begin{theorem}\label{theo:separability}
$\mathcal{M}^{+}(X)$ is non-separable with respect to $\rho$.
\end{theorem}

\begin{proof}
We will exhibit a collection
$\{\mu_{\pmb{x}}\in \mathcal{M}^{+}(X):\ \pmb{x}\in \{0,1\}^{\mathbb{N}}\}$,
such that $\rho(\mu_{\pmb{x}},\mu_{\pmb{y}})>1/2$ whenever $\pmb{x}\neq\pmb{y}$.

\medskip \noindent
Fix $\pmb{x}\in \{0,1\}^\mathbb{N}$, and for each $n\in\mathbb{N}$ and $\pmb{a}\in \{0,1\}^n$ let
\[
q(\pmb{a})=\max\{1\leq k \leq n:\ \pmb{a}_1^k=\pmb{x}_1^k\}+1.
\]
Now, fix $\alpha >1$ and let $\nu_{\pmb{x}}\in \mathcal{M}^{+}(\{0,1\}^{\mathbb{N}})$
be given by
\begin{equation}\label{eq:definitionnu}
\nu_{\pmb{x}}[\pmb{a}]=\left\{\begin{array}{ll}
           \alpha^n(1+\alpha)^{-n}          & \text{ if } \pmb{a}=\pmb{x}_1^n,\\
           \alpha^{q(\pmb{a})-1}(1+\alpha)^{-q(\pmb{a}) }2^{ q(\pmb{a})-n }
                                            & \text{ if } \pmb{a}\neq\pmb{x}_1^n,
                        \end{array}\right.
\end{equation}
for all $n$ and $\pmb{a}\in \{0,1\}^n$.

\medskip \noindent
Let us check that $\nu_{\pmb{x}}$ is well defined. For this notice that
\begin{eqnarray*}
\sum_{\pmb{a}\in \{0,1\}^n}\nu_{\pmb{x}}[\pmb{a}]
     &=&\nu_{\pmb{x}}[\pmb{x}_1^n]+
           \sum_{\pmb{a} \in \{0,1\}^n\setminus\{\pmb{x}_1^n\}}\nu_{\pmb{x}}[\pmb{a}],\\
     & = & \left(\frac{\alpha}{1+\alpha}\right)^n
           +\frac{1}{1+\alpha}\sum_{q=1}^{n}\left(\frac{\alpha}{1+\alpha}\right)^{q-1}
                         \frac{\#\{\pmb{a}\in \{0,1\}^n:\ q(\pmb{a})=m\}}{2^{n-q}},\\
     & = & \left(\frac{\alpha}{1+\alpha}\right)^n +\frac{1}{1+\alpha}
                    \left(\frac{1-(\alpha/(1+\alpha))^n}{1-\alpha/(1+\alpha)}\right)=1,
\end{eqnarray*}
which proves that the marginals are well normalized.
Now, if  $\pmb{a}\in A^{n}$ is such that $q(\pmb{a}) < n$, then
$q(\pmb{a}b)=q(\pmb{a})$ for all $b\in A$, and
\[
\sum_{b\in \{0,1\}}\nu_{\pmb{x}}[\pmb{a}b]=
                        \frac{\alpha^{q(\pmb{a})-1} }{ (1+\alpha)^{q(\pmb{a})} }
                        \frac{2}{2^{n+1-q(\pmb{a})}}=\nu_{\pmb{x}}[\pmb{a}].
\]
Otherwise, if $\pmb{a}=\pmb{x}_1^n$, then
\begin{eqnarray*}
\sum_{b\in A}\nu_{\pmb{x}}[\pmb{a}b]
        &=& \nu_{\pmb{x}}[\pmb{a}x_{n+1}]
                            +\sum_{b\in A\setminus\{x_{n+1}\} }\nu_{\pmb{x}}[\pmb{a}b] \\
        &=&\left(\frac{\alpha}{1+\alpha}\right)^{n+1} +\frac{\alpha^n}{(1+\alpha)^{n+1}}
                       = \left(\frac{\alpha}{1+\alpha}\right)^n = \nu_{\pmb{x}}[\pmb{a}].
\end{eqnarray*}
We have proven that the marginals are well normalized and compatible, 
which ensures that $\nu_{\pmb{x}}$ is well defined.

\medskip \noindent
For $\pmb{y}\neq\pmb{x}$ let  $m=\min\{k\in \mathbb{N}:\ y_k\neq x_k\}$. Then we have
\begin{eqnarray*}
\rho(\nu_{\pmb{x}},\nu_{\pmb{y}}) 
   &\geq& \limsup_{n\to\infty}\frac{1}{n}\left| 
   \log \frac{\nu_{\pmb{x}}[\pmb{x}_1^n]}{\nu_{\pmb{y}}[\pmb{x}_1^n]}\right|,\\
   & = & \limsup_{n\to\infty}\frac{1}{n} \log\left(
                \frac{\alpha^n(1+\alpha)^{-n} }{ \alpha^{q(\pmb{y}_1^n)-1}
                               (1+\alpha)^{-q(\pmb{y}_1^n)}2^{q(\pmb{y}_1^n)-n}}\right),\\
   &=& \lim_{n\to\infty}\frac{1}{n} \log\left(
               \frac{\alpha^n(1+\alpha)^{-n} }{ \alpha^{m-1}(1+\alpha)^{-m}2^{m-n}}\right)
                                              =\log\left(\frac{2\alpha}{1-\alpha}\right).
\end{eqnarray*}
By taking $\alpha=e^{1/2}/(2-e^{1/2})$ we obtain 
$\rho(\nu_{\pmb{x}},\nu_{\pmb{y}})\geq 1/2$ for all $\pmb{x}\neq \pmb{y}$.

\medskip\noindent Now, consider any surjective map $\pi:A\to\{0,1\}$ and for each 
$n\in\mathbb{N}$ extend it coordinatewise to $A^n$. We will denote all those 
coordinatewise extensions with the same letter $\pi$. For each 
$\pmb{x}\in \{0,1\}^{\mathbb{N}}$ the measure $\mu_{\pmb{x}}\in \mathcal{M}^{+}(X)$ 
is given by
\begin{equation}\label{eq:inducedmu}
   \mu_{\pmb{x}}[\pmb{a}]=\frac{\nu_{\pmb{x}}[\pi(\pmb{a})]}{\#\pi^{-1}(\pi(\pmb{a}))}.
\end{equation}
This measure is well defined since for each $n\in \mathbb{N}$
\[
\sum_{\pmb{a}\in A^n}\mu_{\pmb{x}}[\pmb{a}]= \sum_{\pmb{b}\in\{0,1\}^n}
                \#\pi^{-1}(\pmb{b})\frac{\nu_{\pmb{x}}[\pmb{b}]}{\#\pi^{-1}(\pmb{b})}=1,
\]
and for each $\pmb{a}\in A^n$
\begin{eqnarray*}
\sum_{a'\in A}\mu_{\pmb{x}}[\pmb{a}a']&=&
\sum_{a'\in A} \frac{\nu_{\pmb{x}}[\pi(\pmb{a})\pi(a')]}{\#\pi^{-1}(\pi(\pmb{a})\pi(a'))},\\
                                      &=&
\sum_{b\in\{0,1\}} \#\pi^{-1}(b)\frac{\nu_{\pmb{x}}[\pi(\pmb{a})b]
                                    }{\#\pi^{-1}(\pi(\pmb{a}))\,\#\pi^{-1}(b)}=
\mu_{\pmb{x}}[\pmb{a}].
\end{eqnarray*}
Now, for $\pmb{x}\neq\pmb{y}$ we have
\begin{eqnarray*}
\rho(\mu_{\pmb{x}},\mu_{\pmb{y}})
  &=& \sup_{n\in\mathbb{N}}\frac{1}{n}\max_{\pmb{a}\in A^n}
  \left|\log\frac{\mu_{\pmb{x}}[\pmb{a}]}{\mu_{\pmb{y}}[\pmb{a}]}\right|,\\
  &=&\sup_{n\in\mathbb{N}}\frac{1}{n}\max_{\pmb{a}\in A^n} \left|
 \log \frac{\nu_{\pmb{x}}[\pi(\pmb{a})]}{\nu_{\pmb{y}}[\pi(\pmb{a})]} \right|,\\
  &=&\sup_{n\in\mathbb{N}}\frac{1}{n}\max_{\pmb{b}\in \{0,1\}^n}
    \left|\log\frac{\nu_{\pmb{x}}[\pmb{b}]}{\nu_{\pmb{y}}[\pmb{b}]}\right|
                                           =\rho(\nu_{\pmb{x}},\nu_{\pmb{y}})\geq 1/2.
\end{eqnarray*}
In this way we obtain the desired uncountable collection
$\{\mu_{\pmb{x}}\in\mathcal{M}^{+}(X):\ \pmb{x}\in\{0,1\}^{\mathbb N}\}$
such that $\rho(\mu_{\pmb{x}},\mu_{\pmb{y}})\geq 1/2$ whenever 
$\pmb{x}\neq \pmb{y}$.
\end{proof}

\subsection{} According to Equation~\eqref{eq:compare-vage}, the vague topology is 
coarser than the projective topology (the one induce by $\rho$). It is well known, and 
easy to argue, that the $\bar{d}$-topology is finer than the vague topology, and it 
remains to know how to place the projective topology with respect to the $\bar{d}$-
topology. Below we will prove that $\rho$ is not coaser that $\bar{d}$. With this, and a 
construction based on $g$-measures which we will present in 
Section~\ref{sec:g-measures}, we will be able to complete the proof 
that $\rho$ and $\bar{d}$ are not comparable.

\medskip
\begin{theorem}\label{theo:notfiner}
There exists a sequence $\{\mu_p\in\mathcal{M}^{+}(X)\}_{p\in\mathbb{N}}$ 
converging in $\bar{d}$-distance, but not in the projective distance.
\end{theorem}

\medskip
\begin{proof}
Let $\mu_{\pmb{x}}\in \mathcal{M}^{+}(X)$ be as in the proof of
Theorem~\ref{theo:separability}.  We will exhibit a sequence
$\{\pmb{x}_p\in\{0,1\}^{\mathbb{N}}\}_{p\in\mathbb{N}}$ such
$\{\mu_{\pmb{x}_p}\}_{p\in\mathbb{N}}$ converges with respect to $\bar{d}$.

\medskip\noindent
Fix $\pmb{x}\in \{0,1\}^\mathbb{N}$ and for each $p\in\mathbb{N}$ let
$\pmb{x}_p\in\{0,1\}^\mathbb{N}$ be such that
\[
(\pmb{x}_p)_k=\left\{\begin{array}{ll} 1-x_k & \text{ if }  k\in p\,\mathbb{N}+1,\\
                                        x_k  & \text{ if }  k\notin   p\,\mathbb{N}+1.
                     \end{array} \right.
\]
Consider the measures $\mu_{\pmb{x}_p}$ and $\mu_{\pmb{x}}$ as defined in
Equation~\eqref{eq:inducedmu}.
Let us remind that for each $\pmb{y}\in \{0,1\}$, the measure
$\mu_{\pmb{y}}\in \mathcal{M}(X)$ is induced by a corresponding measure
$\nu_{\pmb{y}}\in\mathcal{M}(\{0,1\}^{\mathbb{N}})$, defined in Equation~\eqref{eq:definitionnu},
via a projection $\pi:A\to\{0,1\}$.
Let $\tau:A\to A$ be a permutation satisfying $\tau(a)\in\pi^{-1}(1-\pi(a))$ for each 
$a\in A$ and with this, for each $n\in \mathbb{N}$ define the permutation
$\tau_p:A^n\to A^n$ such that
\[
\tau_p(\pmb{a})_k=\left\{\begin{array}{ll}\tau(a_k)&\text{ if } k\in p\,\mathbb{N}+1,\\
                                               a_k & \text{ if }  k\notin   p\,\mathbb{N}+1.
                               \end{array}\right.
\]
We will denote all those permutations with the same symbol $\tau_p$. With this we 
define the coupling $\lambda_p\in J(\mu_{\pmb{x}_p},\mu_{\pmb{x}})$ such that for 
each $\pmb{a}\times\pmb{b}\in (A\times A)^n$
\[
\lambda_p[\pmb{a}\times\pmb{b}]= \left\{ \begin{array}{ll}
         \mu_{\pmb{x}}[\pmb{a}] &\text{  if  }\pmb{b}=\tau_p(\pmb{a}) ,\\
                              0 &\text{ otherwise.} \end{array} \right.
\]
The permutation $\tau$ is designed so that 
$|a_k-x_k|=|\tau_p(\pmb{a})_k-(\pmb{x}_p)_k|$ for all $1\leq k\leq n$. This ensures 
that $\mu_{\pmb{x}}[\pmb{a}]=\mu_{\pmb{x}_p}[\tau_p(\pmb{a})]$, from which it 
follows that $\lambda_p$ is a coupling. By using this coupling we obtain

\begin{eqnarray*}
\bar{d}(\mu_{\pmb{x}},\mu_{\pmb{x}_p})
  &\leq& \limsup_{n\to\infty}\frac{1}{n}\sum_{k=1}^n\lambda_p(T^{-k}\bar{\Delta})\\
  & = &\limsup_{n\to\infty}\frac{1}{n}\sum_{k=1}^n
         \lambda_p\{\pmb{a}\times\pmb{b}\in (A\times A)^{\mathbb{N})}:\ a_k\neq b_k\}\\
  &  = & \limsup_{n\to\infty}\frac{\#(\{1,2,\ldots,n\}\cap (p\mathbb{N}+1))}{n}=\frac{1}{p}.
\end{eqnarray*}
In this way we have proved that $\mu_{\pmb{x}}=\lim_{p\to\infty}\mu_{\pmb{x}_p}$
in $\bar{d}$-distance.

\medskip\noindent
Theorem \ref{theo:separability} ensures that
$\rho(\mu_{\pmb{x}_p},\mu_{\pmb{x}_{p'}})>1/2$ for all $p\neq p'$. The theorem
follows by taking $\mu_p:=\mu_{\pmb{x}_p}$.
\end{proof}

%%%%%%%%%%%%%%%%%%%%%%%%%%%%%%%%%%%%%%%%%%
%##################################################################
\bigskip

\section{g-measures}
\label{sec:g-measures}
%##################################################################

%##################################################################
\subsection{} Let us start with a brief reminder of $g$-measures. A  
$g$-function is any Borel measurable function $g:X\to (0,1)$ satisfying 
$\sum_{x_1}g(\pmb{x})=1$, 
and a compatible $g$-measure is any 
$\mu\in\mathcal{M}^+_T(X):=\mathcal{M}^+(X)\cap\mathcal{M}_T(X)$ satisfying
\begin{equation}\label{eq:definition-g}
                 \lim_{n\to\infty}\mu(x_1=a_1|\pmb{x}_2^n=\pmb{a}_2^n):=
                 \lim_{n\to\infty}\frac{\mu[a_1\pmb{a}_2^n]}{\mu[\pmb{a}_2^n]}
                                                        =g(\pmb{a}),
\end{equation}
for all $\pmb{a}\in X$. This notion is intended to generalize that of Markov 
chain and was introduced into ergodic theory by M. Keane in~\cite{Keane1972}. 
It has as ancestor the so called chains with complete connections studied in probability 
theory as early as 1935~\cite{Onicescu1935}. This notion is related, and under 
some conditions is equivalent, to the notion of equilibrium 
states~\cite{Walters1975, Keller}. 
One of the main problems concerning $g$-measures is whether a given $g$-function 
admits a unique compatible $g$-measure. Existence of compatible $g$-measures 
requires only the continuity of $g$, while stronger continuity conditions are needed to
ensure uniqueness. 
For instance, H\"older continuity of the $g$-function implies the existence and 
uniqueness of a compatible $g$-measure for which strong mixing holds. Several 
criteria have been established to ensure uniqueness, all of them relying on the 
regularity of the $g$-function. As mentioned in Section~\ref{sec:introduction}, 
several works have considered the $\bar{d}$-continuity of $g$-measures under 
strong regularity conditions for the limit $g$-function, and have proved in this way 
that the limit $g$-measure has good ergodic properties (the Bernoullicity of the natural 
extension~\cite{Coelho1998} or the fast decay of correlation~\cite{Bressaud1999b}). 
On the other hand, several examples have been proposed to show that the 
continuity of the $g$-function is not enough to ensure the uniqueness of the corresponding 
$g$-measure. Among those examples we find the already classical Bramson-Kalikow 
construction~\cite{Bramson1993}. Recently P. Hulse~\cite{Hulse2006} published 
a construction inspired on the Ising model with long range interactions, of a 
$g$-function where uniqueness fails. For this example, the set of compatible 
$g$-measures necessarily contains non-ergodic measures.

\subsection{} Let us start by reminding the notions of variation of a function and 
that of Markov approximation to a measure.

\medskip \noindent
For $\phi:X\to\mathbb{R}$ and each $\ell\in\mathbb{N}$, the $\ell$-variation 
of $\phi$ is given by
\begin{equation}\label{eq:variation}
{\rm var}_\ell \phi:=\max_{\pmb{a}\in A^\ell}\left\{
\sup_{\pmb{x}\in[\pmb{a}]}\phi(\pmb{x})-\inf_{\pmb{x}\in[\pmb{a}]}\phi(\pmb{x})\right\}.
\end{equation}
For $\phi$ continuous we necessarily have $\lim_{\ell\to\infty}{\rm var}_\ell \phi=0$. 
In this case, the speed of convergence of the variation
characterizes the regularity of $\phi$. For instance, H\"older continuity corresponds 
to exponential decreasing of the variation.

\medskip\noindent Given $\mu\in\mathcal{M}(X)$, for each $\ell\in\mathbb{N}$, the 
canonical 
$\ell$-step Markov approximation to $\mu$ is the only measure 
$\mu_{\ell}\in\mathcal{M}(X)$ satisfying 
\begin{equation}\label{Markov-approximation}
\mu_{\ell}[\pmb{a}_1^n]=\mu[\pmb{a}_1^\ell]\prod_{j=1}^{n-\ell}
                                  \frac{\mu[\pmb{a}_j^{j+\ell}]}{\mu[\pmb{a}_{j}^{j+\ell-1}]},
\end{equation} 
for all $\pmb{a}\in X$ and $n\geq \ell$. 

\medskip\noindent It is well known and easily proved that
$\mu_\ell\rightarrow \mu$ as $\ell\to\infty$ in the vague topology. In this respect, 
concerning the $g$-measures, we have the following theorem.

\begin{theorem}\label{theo:Markov-approximations}
Let $g:X\to [0,1]$ be a continuous $g$-function and $\mu\in\mathcal{M}(X)$ a 
compatible $g$-measure. For each $\ell\in\mathbb{N}$ let $\mu_\ell\in\mathcal{M}(X)$ 
be the canonical $\ell$-step Markov approximation. Then $\mu_\ell\rightarrow \mu$ as 
$\ell\to\infty$ in the projective distance. Furthermore,
\[
\rho(\mu_\ell,\mu)\leq {\rm var}_\ell\, \log\circ g.
\]
\end{theorem}

\begin{proof}
First note that for all $\pmb{a}\in X$ and $n\leq m$ we have
\[
\frac{\mu[\pmb{a}_1^n]}{\mu[\pmb{a}_2^n]}=
 \sum_{\pmb{a}_{n+1}^m\in A^{m-n}}\frac{\mu[\pmb{a}_1^m]}{\mu[\pmb{a}_2^m]}
     \times\frac{\mu[\pmb{a}_2^m]}{\mu[\pmb{a}_2^n]}
     =\mathbb{E}_{p}\left(\frac{\mu[\pmb{a}_1^m]}{\mu[\pmb{a}_2^m]}\right),
\]
with $p:A^m\to (0,1)$ a probability distribution given by 
\[
p(\pmb{b})=\left\{\begin{array}{cr}
[\pmb{b}_2^m]/\mu[\pmb{b}_2^n] & \text{ if } \pmb{b}_1^n=\pmb{a}_1^n,\\
0                              & \text{ otherwise}.

\end{array}\right.
\] 
It follows from this, and taking the limit $m\to\infty$, that
\begin{equation}\label{ineq:variation}
\min_{\pmb{x}\in [\pmb{a}_1^\ell]} g(\pmb{x}) \leq 
 \frac{\mu[\pmb{a}_1^n]}{\mu[\pmb{a}_2^n]}  \leq 
\max_{\pmb{x}\in [\pmb{a}_1^\ell]} g(\pmb{x}),
\end{equation}
for all $\pmb{a}\in X$ and $\ell\leq n$.

\medskip\noindent
For $n\leq \ell$ we have $\mu_\ell[\pmb{a}_1^n]=\mu[\pmb{a}_1^n]$ for all 
$\pmb{a}\in X$. On the other hand, for $n > \ell$ and $\pmb{a}\in X$ by writing 
\begin{eqnarray*}
\mu[\pmb{a}_1^n]     &=&\prod_{j=1}^{n-\ell-1}
\frac{\mu[\pmb{a}_j^n]}{\mu[\pmb{a}_{j+1}^n]}\times \mu[\pmb{a}_{n-\ell}^n],\\
\mu_\ell[\pmb{a}_1^n]&=&\prod_{j=1}^{n-\ell-1}
\frac{\mu[\pmb{a}_j^{j+\ell}]}{\mu[\pmb{a}_{j+1}^{j+\ell}]}\times \mu[\pmb{a}_{n-\ell}^n],            
\end{eqnarray*}
we readily obtain
\[
\left|\log\frac{\mu[\pmb{a}_1^n]}{\mu_\ell[\pmb{a}_1^n]}\right|
\leq\sum_{j=1}^{n-\ell-1}
\left\{
 \left|\log\frac{\mu[\pmb{a}_j^n]}{\mu[\pmb{a}_{j+1}^n]}-
       \log\frac{\mu[\pmb{a}_j^{j+\ell}]}{\mu[\pmb{a}_{j+1}^{j+\ell}]}\right|
 \right\}.
\]
Inequalities~\eqref{ineq:variation} imply
\begin{eqnarray*}
\frac{1}{n}\left|\log\frac{\mu[\pmb{a}_1^n]}{\mu_\ell[\pmb{a}_1^n]}\right|
                &\leq & \frac{1}{n}\sum_{j=1}^{n-\ell-1} \left\{
                        \max_{\pmb{x}\in [\pmb{a}_j^{j+\ell}]} \log\circ g(\pmb{x})
                               -\min_{\pmb{x}\in [\pmb{a}_j^{j+\ell}]} \log\circ g(\pmb{x})\right\}\\
               &\leq & {\rm var}_{\ell}\,\log\circ g,  
\end{eqnarray*}
for all $\pmb{a}\in X$ and $n\in \mathbb{N}$, from which it follows that
$\rho(\mu_\ell,\mu)\leq {\rm var}_\ell\, \log\circ g$, and the proof is done.

\end{proof}

\subsection{}  Let us describe the construction by P. Hulse cited above, which we 
slightly modify to fit in our context.  
Consider the real map $t\mapsto \psi(t)=e^t(e^t+e^{-t})^{-1}$ and fix sequences
$\{h_\ell\in\mathbb{R}^+\}_{\ell=0}^\infty$, 
$\{h'_\ell\in\mathbb{R}^+\}_{\ell=0}^\infty$, 
$\{J_\ell\in \mathbb{R}^+\}_{\ell=1}^\infty$, and 
$\{\Lambda_\ell\in\mathbb{N}\}_{\ell=0}^\infty$.
Let $\pi:A\to\{-1,0,1\}$ be such that $\#\pi^{-1}(\{1\})=\#\pi^{-1}(\{-1\})=\lfloor \#A/2\rfloor$. 
With this define the locally constant functions $\{g_\ell, g'_\ell:X\to[0,1]\}_{\ell\in\mathbb{N}}$ 
given by
\begin{eqnarray}\label{Eq:Hulse}
g_\ell(\pmb{x})&=&\psi\left(\beta\,\pi(x_1)
      \left(\sum_{k=1}^\ell J_k\langle\pi(\pmb{x})\rangle_{\Lambda_k}+h_\ell\right)\right),\\
g'_\ell(\pmb{x})&=&\psi\left(\beta\,\pi(x_1)
      \left(\sum_{k=1}^\ell J_k\langle\pi(\pmb{x})\rangle_{\Lambda_k}+h'_\ell\right)\right),
\end{eqnarray}
where 
$\langle\pi(\pmb{x})\rangle_{\Lambda}=\Lambda^{-1}\sum_{m=1}^{\Lambda}\pi(x_k)$ 
for each $\Lambda\in\mathbb{N}$. 
Now, for each $\ell\in\mathbb{N}$, both $g_\ell$ and $g'_\ell$ are constants inside each
cylinder of length $\Lambda_\ell$,  therefore Walters' criterion (logarithm with summable 
variations~\cite{Walters1975}) ensures the existence and uniqueness of $g$-measures 
$\mu_\ell$ and $\mu'_\ell$ compatible with $g_\ell$ and $g'_\ell$ respectively. 
Hulse's construction consist on determining sequences 
$\{h_\ell\in\mathbb{R}^+\}_{\ell=0}^\infty$, $\{h'_\ell\in\mathbb{R}^+\}_{\ell=0}^\infty$, 
$\{J_\ell\in \mathbb{R}^+\}_{\ell=1}^\infty$, and
$\{\Lambda_\ell\in\mathbb{N}\}_{\ell=0}^\infty$, 
ensuring that $\{g_\ell\}_{\ell\in\mathbb{N}}$ and $\{g'_\ell\}_{\ell\in\mathbb{N}}$ have a 
common continuous limit $g:X\to[0,1]$, while $\{\mu_\ell\}_{n\in\mathbb{N}}$ and 
$\{\mu'_\ell\}_{\ell\in\mathbb{N}}$ do not converge to the same measure.
In this way he obtains a simplex (made of all the convex combinations of the two 
different limiting measures) of compatible $g$-measures. 

\medskip\noindent 
From Hulse's construction and Theorem~\ref{theo:Markov-approximations} it readily 
follows the next result. 

\medskip

\begin{theorem}\label{theo:notcoarser}
There exists a sequence $\{\mu_{\ell}\in\mathcal{M}^{+}(X)\}_{\ell\in\mathbb{N}}$ 
converging in the projective distances, but not in the $\bar{d}$-distance.
\end{theorem}

\begin{proof}
Let $g:A\to[0,1]$ be the $g$-function in Hulse's construction above, and let 
${\mathcal M}(g)$ the collection of all the compatible $g$-measures. Since 
${\mathcal M}(g)$ is not a singleton, then it necessarily contains non-ergodic measures,
for instance any strict convex combination of two different extremal measures. 
Let $\mu$ be such a non-ergodic measure.
Now, for each $\ell\in \mathbb{N}$, let $\mu_{\ell}$ be the $\ell$-step Markov 
approximation to $\mu$, as defined in
Equation~\eqref{Markov-approximation}. According to 
Theorem~\ref{theo:Markov-approximations}, the sequence 
$\{\mu_{\ell}\}_{\ell\in\mathbb{N}}$ converges to $\mu$ in the projective distance. 
It is know that $\bar{d}$-limits of mixing measures are mixing (see Theorem I.9.17 
in~\cite{Shields} for instance). Since $\mu$ is fully-supported, then $\mu_{\ell}$ is 
a mixing measure for each $\ell\in\mathbb{N}$ but since $\mu$ is not even ergodic, 
then $\{\mu_{\ell}\}_{\ell\in\mathbb{N}}$ cannot converge in $\bar{d}$-distance.
\end{proof}

%%%%%%%%%%%%%%%%%%%%%%%%%%%%%%%%%%%%%%%

\subsection{} It is know that the entropy is a $\bar{d}$-continuous functional in the 
class of ergodic processes (Theorem I.9.16 in~\cite{Shields}),
while it is only upper semicontinuous with respect to the vague topology (Theorem I.9.1 
in~\cite{Shields}). Concerning the projective distance, we have the following result.

\begin{theorem}\label{theo:entropy-continuity}
Assume $g$ admits a unique $g$-measure $\mu$ (in which case this measure is 
ergodic), and suppose that $\{\mu_p\}_{p\in{\mathbb N}}$ is a sequence of ergodic 
measures converging to $\mu$ in the projective distance, then
\[
\lim_{p\to\infty}h(\mu_p)=h(\mu)\equiv-\int \log\circ g\ d\mu.
\]

\end{theorem}

\begin{proof}
First we prove that the relative entropy
\[h(\mu_p|\mu):=\lim_{n\to\infty}\frac{1}{n}\sum_{\pmb{a}\in A^n}
   \mu_p[\pmb{a}]\log\frac{\mu_p[\pmb{a}]}{\mu[\pmb{a}]},
\] 
which can easily proved to be non-negative, converges to zero as $p\to\infty$. 
Indeed since 
\[
e^{-n\rho(\mu_p,\mu)}\leq
           \frac{\mu_p[\pmb{a}]}{\mu[\pmb{a}]}\leq e^{n\rho(\mu_p,\mu)}
\]  \
for each $n\in {\mathbb N}$ and $\pmb{a}\in A^n$, then
\begin{eqnarray*}
0\leq h(\mu_p|\mu) & = & \lim_{n\to\infty}\frac{1}{n}\sum_{\pmb{a}\in A^n}\mu_p[\pmb{a}]           
                                          \log\frac{\mu_p[\pmb{a}]}{\mu[\pmb{a}]} \\
                               & \leq & \lim_{n\to\infty}\frac{1}{n}\sum_{\pmb{a}\in A^n}
                                    \mu_p[\pmb{a}]\,n\rho(\mu_p,\mu) = \rho(\mu_p,\mu),
\end{eqnarray*}
and the claim follows. Now, following the arguments in \cite[Section 3.2]{chazottes1997}, 
we readily deduce that 
\[
h(\mu_p|\mu)=-h(\mu_p)-\int_X \log\circ g\ d\mu_p.
\]
Now, since the topology of the projective distance is finer than the vague topology, we 
necessarily have
\[
\lim_{p\to\infty}\int_X \log\circ g\ d\mu_p=\int_X \log\circ g\ d\mu. 
\]
Finally, the Variational Principle for $g$-measures (see~\cite{Ledrappier1974} for 
a proof) establishes that
\[
h(\mu)=-\int_X \log\circ g\,d\mu.
\] 
From all the above arguments it follows that
\begin{eqnarray*}
\lim_{p\to\infty}h(\mu)-h(\mu_p)
              &=&\lim_{p\to\infty}\left(-\int_X \log\circ g\,d\mu-h(\mu_p)\right)\\
              &=&\lim_{p\to\infty}\left(-\int_X \log\circ g\,d\mu_p-h(\mu_p)\right)\\
                                                &=&\lim_{p\to\infty}h(\mu_p|\mu)=0,
\end{eqnarray*}
and the proof is done.
\end{proof}

\subsection{} In this paragraph we explore the relationship between convergence of
$g$-functions and the possible convergence in projective distance, of the 
corresponding $g$-measures. An analogous result, concerning the $\bar{d}$-distance, 
was obtained by Coelho and Quas in~\cite{Coelho1998}. Before stating our result, 
let us fix some notation.

\medskip\noindent
Let $\mathcal{G}\subset C_0(X)$ denote the set of $g$-functions, {\it i. e.} the set 
of continuous functions $g: X\to (0,1)$ satisfying 
$\sum_{a\in A} g(a\pmb{x})=1,\ \forall \, \pmb{x}\in X$.
Now, for $g\in\mathcal{G}$ denote by $\mathcal{M}(g)\subset\mathcal{M}(X)$ the simplex 
made of all probability measures compatible with $g$ (or $g$-measures) as defined in
Equation~\eqref{eq:definition-g}. 

\medskip\noindent
For $\phi: X\to \mathbb{R}$ and $N\in\mathbb{N}$, let us denote
${\rm svar}_\ell\phi=\sum_{k=1}^\ell{\rm var}_k\phi$ where ${\rm var}_k \phi$ is 
defined as in Equation~\eqref{eq:variation}. We will say that a locally constant 
function $\phi: X\to\mathbb{R}$ has range $\ell\in\mathbb{N}$ whenever
\[\pmb{x}_1^{\ell}=\pmb{y}_1^\ell \Rightarrow \phi(\pmb{x})=\phi(\pmb{y}).\]
Clearly, for a locally constant function of range $\ell$,
${\rm var}_n\phi= 0$ for all $n\geq \ell$. It is not hard to prove that
if $g\in \mathcal{G}$ is locally constant of range $\ell+1$, then 
$\mathcal{M}(g)$ contains a unique $\ell$-step Markov measure 
(see Section~\ref{sec:locally-constant} for details).
We have the following.

\begin{theorem}\label{theo:g-convergence}
Let $\left\{g_\ell\in \mathcal{G}\right\}_{\ell\in \mathbb{N}}$ be a 
sequence of locally constant functions converging to $g$ in the sup-norm, and
such that for each $\ell\in\mathbb{N}$ the function $g_\ell$ is locally constant 
of range $\ell+1$. If 
\[
\lim_{\ell\to\infty}||\log(g/g_\ell)||e^{{\rm svar}_\ell\log\circ g_\ell}=0,
\]
then the sequences $\left\{\mu_\ell\right\}_{\ell\in \mathbb{N}}$, where 
$\mu_\ell$ is the unique measure in $\mathcal{M}\left(g_\ell\right)$, converges in 
projective distance. Furthermore, the limit measure $\mu\in\mathcal{M}(X)$ is 
the unique measure in $\mathcal{M}(g)$.
\end{theorem}

\begin{proof}
First note that ${\rm var}_m\log\circ g_\ell=0$ and that both $\mu_\ell$ and $\mu_\ell$ 
are $m$-step Markov measures. 
From Proposition~\ref{prop:technical-lemma} in the Appendix, it follows that
\begin{eqnarray*}
\rho(\mu_m,\mu_\ell)&\leq & 2||\log(g_m/g_\ell)||
e^{\min({\rm svar}_\ell\log\circ g_\ell,{\rm svar}_m\log\circ g_m)}\\
                    & \leq &2(||\log(g/g_\ell)||+||\log(g/g_m)||)
e^{\min({\rm svar}_\ell\log\circ g_\ell,{\rm svar}_m\log\circ g_m)}\\
                    &\leq & 2||\log(g/g_\ell)| e^{{\rm svar}_\ell\log\circ g_\ell} +
                        2||\log(g/g_m)| e^{{\rm svar}_m\log\circ g_m},
\end{eqnarray*}
for all $m\geq \ell$. The hypothesis of the theorem implies that  
$\left\{\mu_\ell\right\}_{\ell\in \mathbb{N}}$ is a Cauchy sequence in projective distance, 
and by Theorem~\ref{theo:completeness} it must converge in projective distance
to a certain measure $\mu\in\mathcal{M}^+(X)$.  

\medskip\noindent
Now, since $g=\lim_{\ell\to\infty}g_\ell$ in the sup-norm, then necessarily
$g\in\mathcal{G}$. Let $\nu\in \mathcal{M}(g)$ and for each $\ell\in\mathbb{N}$
let $\nu_\ell$ be its canonical $\ell$-step Markov approximation. Let 
$h_\ell$ be the locally constant $g$-function associate to 
$\nu_\ell$, {\it i. e.}
$h_\ell(\pmb{x})=\nu[\pmb{x}_1^{\ell+1}]/\nu[\pmb{x}_1^{\ell}]$ for all
$\pmb{x}\in X$. According to Inequalities~\eqref{ineq:variation} we have
\[
\min_{\pmb{y}\in [\pmb{x}_1^\ell]}\log\circ g(\pmb{y})\leq \log\circ h_\ell(\pmb{x})
\leq \max_{\pmb{y}\in [\pmb{x}_1^\ell]}\log\circ g(\pmb{y}),
\]
and from this $||\log(g/h_\ell)||\leq {\rm var}_\ell\log\circ g$. Then, using
again Lemma~\ref{prop:technical-lemma} we have
\begin{eqnarray*}
\rho(\mu_\ell,\nu_\ell)&\leq& 
                    2||\log(g_\ell/h_\ell)||e^{{\rm svar}_\ell \log\circ g_\ell}\\
                    &\leq &
2(||\log(g/h_\ell)||+||\log(g_\ell/g)||)e^{{\rm svar}_\ell \log\circ g_\ell}\\
                    &\leq &
2({\rm var}_\ell\log\circ g+||\log(g_\ell/g)||)e^{{\rm svar}_\ell \log\circ g_\ell}.
\end{eqnarray*}
Now, since ${\rm var}_\ell\log\circ g_\ell=0$ and
\[
{\rm var}_\ell\log\circ g\leq {\rm var}_\ell\log\circ g_\ell+||\log(g_\ell/g)||
=||\log(g_\ell/g)||,
\]
it follows that
\[
\rho(\mu_\ell,\nu_\ell)\leq 4||\log(g_\ell/g)||e^{{\rm svar}_\ell \log\circ g_\ell},
\]
which ensures that $\{\nu_\ell\}_{\ell\in\mathbb{N}}$ converges to $\mu$, but 
according to Theorem~\ref{theo:Markov-approximations}, it converges to $\nu$ as well,
therefore $\mu=\nu$ and the proof is finished. 
\end{proof}

\medskip\noindent
\begin{example}\label{exam:long-range}
Consider the sequence of $g$-functions 
$\{g_\ell:\{-1,1\}^{\mathbb{N}}\to (0,1)\}_{\ell\in\mathbb{N}}$ given by
\[
g_\ell(\pmb{x})=\frac{\exp(\beta\,x_1\sum_{k=2}^\ell x_k\,k^{-2})}{
         \exp(+\beta\sum_{k=2}^\ell x_k\,k^{-2})+\exp(-\beta\sum_{k=2}^\ell x_k\,k^{-2})}.
\]
Clearly
$\{g_\ell\}_{\ell\in\mathbb{N}}$ uniformly converges to the 
$g:\{-1,1\}^{\mathbb{N}}\to (0,1)$ given by
\[
g(\pmb{x})=\frac{\exp(\beta\,x_1\sum_{k=2}^\infty x_k\,k^{-2})}{
       \exp(+\beta\sum_{k=2}^\infty x_k\,k^{-2})+\exp(-\beta\sum_{k=2}^\infty x_k\,k^{-2}) }.
\] 
Furthermore, a simple computation leads to the inequalities
\begin{eqnarray*}
||\log(g_\ell/g)||&\leq& 2\beta \sum_{k=\ell+1}^\infty k^{-2} < 2\beta\, \ell^{-1}, \\
\exp({\rm svar}_\ell \log\circ g_\ell)&\leq&\exp(4\beta\sum_{k=2}^\ell (k-1)\,k^{-2} ) < 
\exp(4\beta\,\log(\ell)).
\end{eqnarray*}
According to Theorem~\ref{theo:g-convergence}, the sequence 
$\{\mu_\ell\in\mathcal{M}(g_\ell)\}$ converges in the projective distance to the unique
$g$-measure $\mu\in\mathcal{M}(g_\ell)$, provided
$\ell^{4\beta}\ell^{-1}\to 0$ when $\ell\to\infty$, {\it i. e.}, provided $\beta < 1/4$.
\end{example}

\medskip\noindent 

%%%%%%%%%%%%%%%%%%%%%%%%%%%%%%%%%%%%%%%%%%
\section{Concluding Remarks.}\label{sec:conclusions}
%%%%%%%%%%%%%%%%%%%%%%%%%%%%%%%%%%%%%%%%%%
\noindent With Theorems~\ref{theo:notfiner} and~\ref{theo:notcoarser} we have 
established the incomparability of the $\bar{d}$-topology and the projective topology
in the set of fully-supported probability measures. It is nevertheless not clear 
if this incomparability 
remains in the restriction to the class of invariant probability measures. It is not hard to
verify that the the projective distance between two Markov measures
can be computed by means of a finite algorithm taking the 
parameters defining the measures as inputs. One can also argue that the output value 
varies continuously or at worst piecewise continuously with the input parameters. 
This this does not seem to be the case of the $\bar{d}$ distance,
which suggests that in the class of Markov measures the projective topology is coarser  
than the $\bar{d}$ topology.

\medskip \noindent 
Theorem~\ref{theo:g-convergence} establishes a new criterion for uniqueness 
of $g$-measures based on the speed of convergence of locally constant approximations
to the $g$-function. It can be related to a similar criterion ensuring convergence in 
$\bar{d}$-distance established by Coelho and Quas in~\cite{Coelho1998}. Although in
our case we cannot deduce that the limit measure satisfies the Bernoulli property, we can 
nevertheless ensure that the limit measure inherits the mixing property of the Markov 
approximations, and thanks to Theorem~\ref{theo:entropy-continuity}, that the
the entropy is continuous with respect to the projective distance at the limit measure. 

\medskip \noindent Example~\ref{exam:long-range} is the $g$-measure analog of the 
one-dimension Ising model with long range interaction, for which a phase transition has been 
proved to occur (see~\cite{Dyson1969, Frohlich1982} for details). The analogy suggests
that the uniqueness of the associated $g$-measure must break at high values of the 
parameter $\beta$. This transition should be detectable through a criterion involving 
the regularity of the $g$-function and the speed of convergence of the Markov 
approximations. 

\medskip \noindent The projective distance appears to be suited for the study of measures
obtained by random substitutions as the one we have characterized in~\cite{Salgado2013}. 
We can prove that for a certain class of random substitutions, the substitution process is a 
contraction in the projective distance, and that the unique attractor has the mixing property. 
The study of this kind of processes and its characterization in terms of the projective distance 
is the subject of a forthcoming work.

%%%%%%%%%%%%%%%%%%%%%%%%%%%%%%%%%%%%%%%%%%
\appendix
\section{}

\subsection{}\label{sec:locally-constant}
A $n\times n$ real matrix $M$ is said to be {\it primitive} if $M\geq 0$ 
({\it i. e.} none of its entries is negative) and for some $k\in\mathbb{N}$, 
$M^k>0$ ({\it i. e.} all the entries of $M^k$ are positive). 
The {\it primitivity index} of a primitive matrix $M$ is the smallest integer 
$\ell$ such that $M^\ell>0$. The Perron-Frobenius Theorem ensures that the spectral 
radius ({\it i. e.} the maximal norm of its eigenvalues) of a primitive matrix 
$M$ is achieved by a simple positive eigenvalue $\lambda$ with positive right 
and left eigenvectors $\pmb{v}$ and $\pmb{w}$ respectively.

\medskip\noindent The function $d_p(0,\infty)^n\times(0,\infty)^n\to [0,\infty)$ 
such that
\begin{equation}\label{eq:projective-pseudodistance}
d_p(\pmb{x},\pmb{y}):=\max_{1\leq i\leq n}\log\frac{x_i}{y_i}-
                                 \min_{1\leq i\leq n}\log\frac{x_i}{y_i},
\end{equation}
defines {\it a projective pseudo-distance} which becomes a distance when restricted 
to the simplex of probability vectors. A refined version of the Perron-Frobenius 
Theorem which we can find in~\cite{Seneta}, establishes that the action of a 
$n\times n$ primitive matrix $M$ with primitivity index $\ell$, over the cone 
$(0,\infty)^n$ defines a contraction with respect to the projective pseudo-distance
$d_p$. More precisely, for all $\pmb{x},\pmb{y}\in (0,\infty)^n$ we have
\begin{equation}\label{ineq:projective-contraction} 
d_p(M\pmb{x},M\pmb{y})\leq d_p(\pmb{x},\pmb{y})\, \text{ and }\, 
d_p(M^\ell\pmb{x},M^\ell\pmb{y})\leq \tau_M d_p(\pmb{x},\pmb{y}),
\end{equation}
where
\begin{equation}\label{eq:birkhoff-coefficient}
\tau_M=\frac{1-\sqrt{\min_{i,j,k,l} \frac{M^\ell(i,j)M^\ell(k,l)}
{M^\ell(i,l)M^\ell(k,j)}}}{
                     1+\sqrt{\min_{i,j,k,l} \frac{M^\ell(i,j)M^\ell(k,l)
                                                }{M^\ell(i,l)M^\ell(k,j)}}}.
\end{equation}
The coefficient $\tau_M$ is the so called {\it Birkhoff's contraction coefficient}. 

\begin{proposition}\label{prop:projective-contraction}
Let $P,Q:\{1,2,\ldots, n\}\times\{1,2,\ldots, n\}\to (0,1)$ be  stochastic by 
columns, {\it i. e.}, $\sum_{i=1}^n P(i,j)=\sum_{i=1}^n Q(i,j)=1$ for each 
$j\in\{1,2,\ldots,n\}$. Suppose that 
\[
e^{-\epsilon} \leq P(i,j)/Q(i,j) \leq e^{\epsilon}
\]
for some $\epsilon >0$ and each $i,j\in\{1,2,\ldots,n\}$. Then the maximal 
eigenvalue of both matrices is 1, and the associated positive right eigenvectors 
$u, v$ are such that
\[
          d_p(u,v)\leq \frac{\epsilon}{1-\min(\tau_P,\tau_Q)},
\]
where $\tau_P$ and $\tau_Q$ are the Birkhoff coefficients of $P$ and $Q$ 
respectively.
\end{proposition}

\begin{proof}
First note that a $n\times n$ positive matrix $M$, stochastic by columns, preserves 
the simplex of probability vectors $\Delta=\{u\in [0,1]^n: \ \sum_{i=1}^n u(i)=1\}$. 
Therefore, according to Inequality~\eqref{ineq:projective-contraction} and 
Banach's fixed point Theorem, the transformation $u\mapsto Mu$ has a unique fixed 
point $v\in \Delta$, which necessarily coincides with a positive eigenvector of $M$ 
associated to the eigenvalue 1. Furthermore, because of the contractiveness of $M$
with respect to $d_p$, we have $v=\lim_{n\to\infty} M^n u$ for all $u\in\Delta$. 
Hence there cannot be another positive eigenvector which implies that 1 necessarily 
is the maximal eigenvalue of $M$. In this way we prove in particular that 1 is the
maximal eigenvalue of both $P$ and $Q$ with unique eigenvectors $u, v\in \Delta$ 
respectively.

\medskip\noindent
Let us assume now that $\tau_Q \leq \tau_P$, then
\begin{eqnarray*}
d_p(u,v)&\leq &\lim_{N\to\infty} \sum_{n=0}^N d_p(Q^n u,Q^{n+1} u) +d_p(Q^{N+1},v),\\
       &\leq &d_p(u,Q u)\sum_{n=1}^\infty\tau_Q^n= 
             \frac{d_p(u,Q u)}{1-\tau_Q}  =  \frac{d_p(Pu,Q u)}{1-\tau_Q}.
\end{eqnarray*}
Finally, since $e^{-\epsilon} \leq P(i,j)/Q(i,j) \leq e^{\epsilon}$ for all 
$i,j\in \{1,2,\ldots, n\}$, then
\[
e^{-\epsilon}\leq \frac{\sum_{k=1}^n P(i,j)u(j)}{\sum_{k=1}^n Q(i,j)u(j)}
\leq e^{\epsilon}
\]
for all $1\leq i\leq n$, and from this
\[
d_p(Pu,Qu)=\max_{1\leq i\leq n} \log\frac{(Pu)(i)}{(Qu)(i)}
              -\min_{1\leq i\leq n} \log\frac{(Pu)(i)}{(Qu)(i)}\leq 2\epsilon.
\]

\end{proof}

\subsection{}
To a $\ell$-step Markov measure $\mu\in\mathcal{M}^+(X)$ it corresponds a
locally constant $g$-function $g_\mu: X\to (0,1)$ given by
\[
g_\mu(\pmb{x})=\frac{\mu[\pmb{x}_1^{\ell+1}]}{\mu[\pmb{x}_2^{\ell+1}]},
\]
and such that $\mu$ is the unique $g_\mu$-measure, {\it i. e.} 
$\mathcal{M}(g_\mu)=\{\mu\}$. The function $g_\mu$ defines a primitive matrix 
$M_\mu:A^{\ell}\times A^{\ell}\to [0,1]$ as follows:
\begin{equation}\label{eq:primitive-matrix}
M_{\mu}\left(\pmb{a}_1^{\ell},\pmb{b}_1^{\ell}\right)
   =\left\{\begin{array}{cr}
  g_\mu (\pmb{a}b_\ell )  & \text{ if } \pmb{a}_2^\ell=\pmb{b}_1^{\ell-1},\\
     0                    & \text{ otherwise}.
                        \end{array}\right.
\end{equation}
It is easily verified that $M_\mu^{\ell}>0$ and that 1 is $M_\mu$'s maximal 
eigenvalue with right eigenvector $v: A^{\ell}\to (0,1)$ such that 
$v(\pmb{a})=\mu[\pmb{a}]$. From Proposition~\ref{prop:projective-contraction} 
we derive the following.

\begin{proposition}\label{prop:technical-lemma}
Let $\mu,\nu \in \mathcal{M}^+(X)$ be two $\ell$-step Markov measures, 
and let $g_\mu,g_\nu\in \mathcal{G}$ be the locally constant $g$-functions 
associated to $\mu$ and $\nu$ respectively. Then
\[
\rho(\mu,\nu)\leq 2||\log (g_\mu/g_\nu )|| e^{\min({\rm svar}_\ell g_\mu,
                                                       {\rm svar}_\ell g_\nu)}.
\]
\end{proposition}

\begin{proof}
Let $v_\mu$ be such that $v_{\mu}(\pmb{a})=\mu[\pmb{a}]$ for all 
$\pmb{a}\in A^{\ell}$, and similarly for $v_\nu$. Then, 
Proposition~\ref{prop:projective-contraction} directly implies that
\[
d_p(v_\mu,v_\nu)\leq \frac{2||\log (g_\mu/g_\nu )||}{1-\min(\tau_\mu,\tau_\nu)}.
\]
It can be easily verified that 
$\tau_\mu < 1-\exp\left(-{\rm svar}_\ell\log\circ g_\mu\right)$, and similarly 
for $\tau_\nu$. From this it follows that
\[
d_p(v_\mu,v_\nu)\leq  2||\log (g_\mu/g_\nu )|| 
 e^{\min({\rm svar}_\ell \log\circ g_\mu,{\rm svar}_\ell  \log\circ g_\nu)}.
\]
Let us remind that 
$\rho(\mu,\nu)= \sup_{N\in\mathbb{N}} \max_{\pmb{a}\in A^N}\left|
\log(\mu[\pmb{a}]/\nu[\pmb{a}])\right|/N$.
If the supreme is not reached at $N < \ell$, then
\begin{eqnarray*}
\rho(\mu,\nu)&=&\sup_{N\in\mathbb{N}} \max_{\pmb{a}\in A^N}\left|\frac{1}{N}
  \sum_{N=1}^{n-\ell}
  \log\left(\frac{g_\mu\left(\pmb{a}_n^{n+\ell}\right)}{g_\nu\left(\pmb{a}_n^{n+\ell}\right)}\right)
  +\frac{1}{N}\log\frac{\mu[\pmb{a}_{N-\ell+1}^N]}{\nu[\pmb{a}_{N-\ell+1}^N]}\right| \\
   &=&\sup_{N\in\mathbb{N}} \max_{\pmb{a}\in A^N}\left|\frac{1}{N}
  \sum_{N=1}^{n-\ell}
  \log\left(\frac{g_\mu\left(\pmb{a}_n^{n+\ell}\right)}{g_\nu\left(\pmb{a}_n^{n+\ell}\right)}\right)
  +\frac{1}{N}\log\frac{v_\mu\left(\pmb{a}_{N-\ell+1}^N\right)}{
                         v_\nu\left(\pmb{a}_{N-\ell+1}^N\right)}\right|\\
  &\leq&\max\left(||\log(g_\mu/g_\nu)||, ||\log(v_\mu/v_\nu)||\right).
\end{eqnarray*}
On the other hand, if the supreme is achieved at some $N<\ell$ then
\begin{eqnarray*}
\rho(\mu,\nu)&\leq& \max_{\pmb{a}\in A^N}\frac{1}{N}
 \left|\log\left( \frac{\sum_{\pmb{b}\in A^{N-\ell} } v_\mu(\pmb{a}\pmb{b})
                            }{ \sum_{\pmb{c}\in A^{N-\ell} } 
                            v_\nu(\pmb{a}\pmb{c})}\right)\right| \\
                      &\leq&\max_{\pmb{a}\in A^N}
                      \left|\log\left(\sum_{\pmb{b}\in A^{N-\ell} } 
                      \frac{v_\mu(\pmb{a}\pmb{b})}{v_\nu(\pmb{a}\pmb{b})}
                      \times \frac{v_\nu(\pmb{a}\pmb{b})
                                      }{\sum_{\pmb{c}\in A^{N-\ell} } 
                                      v_\nu(\pmb{a}\pmb{c})}\right)\right|\\
                       & \leq &\max_{\pmb{a}\in A^N}\left|
                            \log\max_{\pmb{b}\in A^{N-\ell}}
                       \frac{v_\mu(\pmb{a}\pmb{b})}{v_\nu(\pmb{a}\pmb{b})}\right|
                            =||\log(v_\mu/v_\nu)||.
\end{eqnarray*}
Finally, since both $v_\mu$ and $v_\nu$ are probability vectors, we have
\[
||\log(v_\mu/v_\nu)||\leq 
\max_{\pmb{a}\in A^\ell}\log\frac{v_\nu(\pmb{a})}{v_\nu(\pmb{a})}
  -\min_{\pmb{a}\in A^\ell}\log\frac{v_\nu(\pmb{a})}{v_\nu(\pmb{a})} 
   \equiv d_p(v_\mu,v_\nu),
\]
and with this
\begin{eqnarray*}
\rho(\mu,\nu)&\leq& \max\left(||\log(g_\mu/g_\nu)||, d_p(v_\mu,v_\nu)\right)\\
                     &\leq& 
2||\log (g_\mu/g_\nu )|| e^{\min({\rm svar}_\ell  \log\circ g_\mu,{\rm svar}_\ell  \log\circ g_\nu)}.
\end{eqnarray*}
\end{proof}

%%%%%%%%%%%%%%%%%%%%%%%%%%%%%%%%%%%%%%%%%%
\bigskip
\bibliography{TrejoUgalde2015}
\bibliographystyle{plain}
%%%%%%%%%%%%%%%%%%%%%%%%%%%%%%%%%%%%%%%%%%
\end{document}